\documentclass[onefignum,onetabnum]{siamart171218} 

\usepackage{lipsum}
\usepackage{amsfonts}
\usepackage{graphicx}
\usepackage{epstopdf}
\usepackage{algorithmic}
\ifpdf
  \DeclareGraphicsExtensions{.eps,.pdf,.png,.jpg}
\else
  \DeclareGraphicsExtensions{.eps}
\fi

\newsiamremark{remark}{Remark}
\newsiamremark{example}{Example}
\newsiamremark{hypothesis}{Hypothesis}
\crefname{hypothesis}{Hypothesis}{Hypotheses}
\newsiamthm{claim}{Claim}

\headers{A Multilevel Approach to Stochastic Estimation of the Trace}{A. Frommer, M. Nasr Khalil and G. Ramirez-Hidalgo}

\title{A Multilevel Approach to Variance Reduction in the Stochastic Estimation of the Trace of a Matrix\thanks{Submitted to the editors DATE.
\funding{This work was funded by the European Commission within the STIMULATE project. }}}

\author{Andreas Frommer\thanks{University of Wuppertal, Department of Mathematics, Gauss-Strasse 20, 42097 Wuppertal, Germany (\email{frommer@uni-wuppertal.de}, \email{m.nasr@stimulate-ejd.eu}, \email{g.ramirez@stimulate-ejd.eu} ).}
\and Mostafa Nasr Khalil\footnotemark[2]
\and Gustavo Ramirez-Hidalgo\footnotemark[2]}

\usepackage{amsopn}
\usepackage{graphicx}
\DeclareMathOperator{\diag}{diag}

\makeatletter
\newcommand*{\addFileDependency}[1]{
  \typeout{(#1)}
  \@addtofilelist{#1}
  \IfFileExists{#1}{}{\typeout{No file #1.}}
}
\makeatother

\ifpdf
\hypersetup{
  pdftitle={A Multilevel Approach to Stochastic Estimation of the Trace},
  pdfauthor={A. Frommer, M. Nasr Khalil and G. Ramirez-Hidalgo}
}
\fi

\usepackage{svg}
\usepackage{amsmath}

\newcommand{\mathC}{\mathbb{C}}
\newcommand{\mF}{\mathcal{F}}
\renewcommand{\P}{\mathbb{P}}
\newcommand{\E}{\ensuremath{\mathbb{E}}}
\newcommand{\V}{\ensuremath{\mathbb{V}}}
\newcommand{\tr}{\ensuremath{\mathrm{tr}}}

\newcommand{\C}{\mathbb{C}}

\renewcommand{\i}{\ensuremath{\mathrm{i}}}

\newcommand{\defl}{\ensuremath{\mathrm{defl}}}
\newcommand{\nnz}{\ensuremath{\mathrm{nnz}}}
\newcommand{\offdiag}{\ensuremath{\mathrm{offdiag}}}

\newcommand{\comment}[1]{}

\usepackage{tabto}

\begin{document}

\maketitle

\begin{abstract}
The trace of a matrix function $f(A)$, most notably of the matrix inverse, can be estimated stochastically using samples $x^*f(A)x$ if the components of the random vectors $x$ obey an appropriate probability distribution. 
However such a Monte-Carlo sampling suffers from the fact that the accuracy depends quadratically of the samples to use, thus making higher precision estimation very costly. In this paper we suggest and investigate a multilevel Monte-Carlo approach which uses a multigrid hierarchy to stochastically estimate the trace. This results in a substantial reduction of the variance, so that higher precision can be obtained at much less effort. We illustrate this for the trace of the inverse using three different classes of matrices. 
\end{abstract}

\begin{keywords}
 Monte-Carlo, multilevel Monte-Carlo, trace, Hutchinson's method, sparse matrices  
 \end{keywords}

\begin{AMS}
  65C05, 65F10, 65F60, 65N55
\end{AMS}

\section{Introduction}



We consider the situation where one is interested in 
the trace $\tr(f(A))$ of a matrix function $f(A)$. Here, $f(A) \in \mathC^{n \times n}$ is the matrix obtained from 
$A \in \mathC^ {n \times n}$ and the function $f: z \in D\subseteq \mathC \to f(z) \in \mathC$ in the usual operator theoretic sense; see \cite{Higham2008}, e.g. Our focus is on the inverse $A^{-1}$, i.e.\ $f(z) = z^{-1}$. Computing the trace is an 
important task arising in many applications. The trace of the inverse is required, for example, in the study of 
fractals \cite{Sapoval1991}, in generalized cross-validation and its applications 
\cite{GolubMatt1995,GolubHeathWahba1979}. In network analysis, the Estrada index---a total centrality measure for networks---is defined as the trace of the exponential of the adjacency matrix $A$ of a graph 
\cite{EstradaHigham2010,Ginosar2008} and an analogous measure is given by the trace of the resolvent  $(\rho I-A)^{-1}$ \cite[Section 8.1]{Estrada2011}. For Hermitian positive definite matrices $A$, one can compute the 
log-determinant $\log(\det(A))$ as the trace of the logarithm of $A$. The log-determinant is needed in machine 
learning and related fields \cite{Rasmussen2005,Rue2005}. Further applications are discussed in 
\cite{meyer2021hutch++,SaadUbaru2017,SaadUbaruJie2017}. A particular application which we want to prepare with this work is
in lattice quantum chromodynamics (QCD), a computational approach in Theoretical Physics to simulate the interaction of the quarks as constituents of matter. Here, the trace of the inverse of the discretized Dirac operator yields the disconnected fermion loop contribution to an observable; see \cite{SextonWeingarten1994}. As  simulation methods get more and more precise, these contributions become increasingly important. 

It is usually unfeasible to compute the diagonal entries $f(A)_{ii}$ directly as $e_i^*f(A)e_i$, $e_i$
the $i$th canonical unit vector, and then obtain the trace by summation. For example, for the inverse this would mean that we have to solve $n$ linear systems, which is prohibitive for large values of $n$. 

One large class of methods which aims at circumventing this cost barrier are deterministic approximation techniques. Probing methods, for example, approximate
\begin{equation} \label{eq:probing}
\tr(f(A)) \approx \sum_{i=1}^N w_i^* f(A) w_i,
\end{equation}
where the vectors $w_i$ are carefully chosen sums of canonical unit vectors and $N$ is not too large. Various approaches have been suggested and explored in order to keep $N$ small while at the same time achieving good accuracy in \eqref{eq:probing}. This includes approaches based on graph colorings; see \cite{BekasKokiopoulouSaad2007,Endress:2014qpa,TangSaad2012} e.g., and the hierarchical probing techniques from \cite{Stathopoulos2013,LaeuchliStathopoulos2020}. In order for probing to yield good results, the matrix $f(A)$ should expose a decay of the moduli of its entries when we move away from the diagonal, since the sizes of the entries farther away from the diagonal determine the accuracy of the approximation. Recent theoretical results in this direction were given in \cite{FrSchiSchw2020}.  Lanczos techniques represent another deterministic approximation approach and are investigated  \cite{Bentbib2021,JieSaad2018,Lin2016}, e.g. Without giving details let us just mention that in order to improve their accuracy, deterministic approximation techniques can be combined with the stochastic techniques to be presented in the sequel; see \cite{SaadUbaruJie2017}, e.g.

In this paper we deal with the other big class of methods which aim at breaking the cost barrier using {\em stochastic estimation}.  In principle, they work for any matrix and, for example, do not require a decay away from the diagonal. Our goal is to develop a multilevel Monte-Carlo method to estimate $\tr(f(A))$ stochastically. Our approach can be regarded as a variance reduction technique applied to the classical stochastic ``Hutchinson'' estimator \cite{Hutchinson90}
\begin{equation} \label{hutch:eq}
\tr(f(A)) \approx \frac{1}{N} \sum_{n=1}^N x^{(n)} f(A) x^{(n)},
\end{equation}
where the components of the random vectors $x^{(n)} $ obey an appropriate probability distribution. The variance of the estimator \eqref{hutch:eq} decreases only like $\frac{1}{N}$, which makes the method too costly when higher precisions are to be achieved. The multilevel approach aims at curing this by working with representations of $A$ at different levels. On the higher numbered levels, evaluating $f(A)$ becomes increasingly cheap, while on the lower levels, which are more costly to evaluate, the variance is small.

This paper is organized as follows: In Section~\ref{mlmc:sec} we recall the general 
framework of multilevel Monte-Carlo estimators. In 
Section~\ref{trace_est:sec} we discuss Hutchinson's method for stochastically estimating the trace before turning to our new multilevel approach in Section~\ref{mlmc_trace:est}. This 
section also contains a comparison to known approaches based on deflation as a motivation of why the new multilevel method should provide additional efficiency. Finally, several numerical results are presented in Section~\ref{numerical_results:sec}.

\section{Multilevel Monte-Carlo} \label{mlmc:sec}
We discuss the basics of the multilevel Monte-Carlo approach as a variance reduction technique. We place ourselves in a general setting, thereby closely following \cite{Giles2015}.

Assume that we are given a probability space $(\Omega, \mF, \P)$ with sample space $\Omega$, sigma-algebra
$\mF \subseteq \Omega$ and probability measure $\P: \mF \to [0,1]$. For a given random variable 
$f: \Omega \to \C$ , the {\em standard Monte-Carlo} approach estimates its expected value $\E[f]$ as the arithmetic mean
\begin{equation}  \label{stdMC:eq}
\E[f] \approx \frac{1}{N} \sum_{n=1}^N f(\omega^{(n)}),
\end{equation}
where the $\omega^{(i)}$ are independent events coming from $(\Omega, \mF, \P)$. 
The variance of this estimator is $\frac{1}{N} \V[f]$, so the root mean square  deviation has order $\mathcal{O}(N^{-1/2})$. This indicates that the number $N$ of events has to increase quadratically with the accuracy required which is why, typically, higher accuracies require very high computational effort in this type of Monte-Carlo estimation.

The idea of {\em multilevel Monte-Carlo} is to split the random variable $f$ as a sum 
\begin{equation} \label{multilevel_decomposition:eq}
f = \sum_{\ell=1}^L g_{\ell},
\end{equation}
where the random variables $g_{\ell}:\Omega \to \C$ are regarded as contributions ``at level $\ell$'' to $f$. This gives 
\[
\E[f] =  \sum_{\ell=1}^L \E[g_{\ell}],
\]
and an unbiased estimator for $\E[f]$ is obtained as 
\[
\E[f] \approx  \sum_{\ell=1}^L \frac{1}{N_\ell} \sum_{n=1}^{N_\ell}   g_{\ell}(\omega^{n,\ell}),
\]
where the $\omega^{(n,\ell)}$ denote the independent events on each level. 
The variance of this estimator is 
\[
\sum_{\ell=1}^{L} \frac{1}{N_\ell} \V[g_\ell].
\]
The idea is that we are able to find a multilevel decomposition of the form \eqref{multilevel_decomposition:eq} in which the cost $C_\ell$ to evaluate $g_\ell$ is low when the variance $V_\ell := \V[g_\ell]$ is high  and vice versa. 
As is explained in \cite{Giles2015}, the solution 
to the minimization problem which minimizes the total cost subject to achieving a given target variance $\epsilon^2$
\[
\sum_{l=1}^L N_\ell C_\ell \to \min! \enspace \mbox{s.t. }  \sum_{\ell=1}^{L} \frac{1}{N_\ell} V_\ell = \epsilon^2
\]
gives $N_\ell = \mu \sqrt{V_\ell/C_\ell}$. Here, the Lagrangian multiplier $\mu$ satisfies $\displaystyle \mu  =  \epsilon^{-2} \sum_{\ell=1}^L \sqrt{V_\ell/C_\ell}$, and the corresponding minimal total cost is
\[
C = \epsilon^{-2} \left( \sum_{\ell = 1}^L \sqrt{V_\ell C_\ell} \right)^2.
\]

The typical situation is that the contributions $g_\ell$ on level $\ell$ are given as
differences $f_\ell - f_{\ell+1}$ of approximations $f_\ell$ to $f$ on the various levels, i.e.\ we have
\begin{equation} \label{mlmc_representation:eq}
f = \sum_{\ell=1}^{L-1} \underbrace{\left(f_{\ell}-f_{\ell+1}\right)}_{=g_{\ell}} + \underbrace{f_L}_{=g_L} \quad \text{with } f_1 = f.
\end{equation}
 If we assume that the cost $\hat{C}_\ell$ to evaluate $f_\ell$ decreases rapidly with the level $\ell$, the cost ${C}_\ell$ for evaluating the differences $g_\ell = f_\ell-f_{\ell+1}$ is well approximated by $\hat{C}_\ell$. The 
 ratio of the total cost encountered when reducing the variance to a given value between multilevel Monte-Carlo (with optimal choice of $N_\ell$) and standard Monte-Carlo  \eqref{stdMC:eq} is then given by
\[
\left( \sum_{\ell = 1}^L \sqrt{V_\ell C_\ell}\right)^2 \; \big/ \;  \left(\V[f]C_1\right) .
\]
This is the basic quantitative relation indicating how the costs $C_\ell$ to evaluate the $f_\ell$ and the variances $V_\ell$ of the differences $f_\ell-f_{\ell+1}$ have to relate in order for the multilevel approach to be more efficient than standard Monte-Carlo estimation of $f$.

\section{Stochastic estimation of the trace of a matrix} \label{trace_est:sec}
We now assume that we are given, in an indirect manner, a matrix $A = (a_{ij}) \in \C^{n \times n}$ for which we want to compute the trace
\[
\tr(A) = \sum_{i=1}^n a_{ii}. 
\]
Our basic assumption is that the entries $a_{ii}$ of $A$ are neither available directly nor that they can all be obtained at decent computational cost. This is typically the case when $A$ arises as a function of a large (and sparse) matrix, the most common case being the matrix inverse.

In a seminal paper \cite{Hutchinson90}, Hutchinson suggested to use a stochastic estimator to approximate $\tr(A)$. The following theorem summarizes his result together with the generalizations on the admissible probability spaces; see \cite{DongLiu1993,Wilcox99}, e.g.

\begin{theorem} \label{hutchinson:thm}
Let $\P: \Omega \to [0,1]$ be a probability measure on a sample space $\Omega$ and assume that the components $x_i$ of the vector $x \in \mathC^n$ are random variables depending on $\omega \in \Omega$ satisfying
\begin{equation} \label{P_assumptions:eq}
 \E[x_i] = 0 \enspace \mbox{ and } \enspace
 \E[\overline{x}_ix_j] =  \delta_{ij} \enspace \mbox{ (where $\delta_{ij}$ is the Kronecker delta)}.
\end{equation}
Then
\[
   \E[x^*Ax] = \tr(A) \mbox{ and }
    \V[x^*Ax] = \sum_{\stackrel{i,j,k,p=1}{i\neq j, k \neq p}}^n \overline{a}_{ij}a_{kp} \E[  \overline{x}_ix_j\overline{x}_kx_p].
\]
In particular, if the probability space is such that each component $x_i$ is independent  from $x_j$ for $i \neq j$, then
\[
\V[x^*Ax] =  \sum_{\stackrel{i,j}{i \neq j}}^n\overline{a}_{ij}a_{ij} 
+ 
\sum_{\stackrel{i,j}{i \neq j}}^n\overline{a}_{ij}a_{ji} \E[x_i^2]\E[\overline{x}_j^2]. 
\]
\end{theorem}
\begin{proof} The proof is simple, but we repeat it here because the literature often treats only the real and not the general complex case. We have
\[
\E[x^*Ax] = \sum_{i=1}^n a_{ii}\E(\overline{x}_ix_i) + \sum_{i,j=1, i\neq j}^n a_{ij}\E(\overline{x}_ix_j) = \tr(A),
\]
where the last inequality follows from \eqref{P_assumptions:eq}. Similarly
\begin{eqnarray}
\V[x^*Ax] &=& \E \left[ \overline{(x^*Ax-\tr(A))}(x^*Ax-\tr(A))\right] \nonumber \\
          &=& \E \big[ \big( \sum_{\stackrel{i,j=1}{i\neq j}}^n x_i\overline{a}_{ij} \overline{x}_j\big)  \big( \sum_{\stackrel{k,p=1}{k\neq p}}^n \overline{x}_k a_{kp} x_p\big)\big] \nonumber \\
          &=& \E \big[  \sum_{\stackrel{i,j,k,p=1}{ i\neq j, k\neq p}}^n \overline{a}_{ij}a_{kp}x_i\overline{x}_j \overline{x}_k  x_p\big] 
          \, = \,  \sum_{\stackrel{i,j,k,p=1}{i \neq j, k \neq p}}^n\overline{a}_{ij}a_{kp} \E[x_i\overline{x}_j \overline{x}_k  x_p]
          \label{sum_prod_expected:eq}.
\end{eqnarray}
Since the components $x_i$ are assumed to be independent, we have $\E[\overline{x}_ix_j \overline{x}_k  x_p]
= 0$ except when $i =j, k=p$ (which does not occur in \eqref{sum_prod_expected:eq}) or $i=k, j=p$ or $i=p, j=k$. This gives
\[
\sum_{\stackrel{i,j,k,p=1}{i \neq j, k \neq p}}^n\overline{a}_{ij}a_{kp} \E[x_i\overline{x}_j \overline{x}_k  x_p] 
=
\sum_{\stackrel{i,j}{i \neq j}}^n\overline{a}_{ij}a_{ij} \E[x_i\overline{x}_j \overline{x}_i  x_j] 
+ 
\sum_{\stackrel{i,j}{i \neq j}}^n\overline{a}_{ij}a_{ji} \E[x_i\overline{x}_j \overline{x}_j  x_i], 
\]
and in the first sum $\E[x_i\overline{x}_j \overline{x}_i  x_i] = \E[\overline{x}_ix_i] \E[\overline{x}_jx_j] = 1$ by assumption, whereas in the second sum we have $ 
\E[x_i\overline{x}_j \overline{x}_j  x_i] = \E[x_i^2]\E[\overline{x}_j^2]$.
\end{proof}

Note that as a definition for the variance of a complex random variable $y$ we used $\E[(\overline{y-E(y})(y-E[y]]$ 
rather than $\E[(y-E[y])^2] $ to keep it real and non-negative.

Standard choices for the probability spaces are to take $x$ with identically and independently distributed (i.i.d.) components as
\begin{align}
\label{iid_components1:eq}    & x_i \in \{-1,1\} \mbox{ with equal probability } \tfrac{1}{2}, \\
\label{iid_components2:eq}    &x_i \in \{-1,1,-i,i\} \mbox{ with equal probability } \tfrac{1}{4}, \\
\label{iid_components3:eq}    & x_i = \exp(i\theta) \mbox{ with $\theta$ uniformly distributed in  $[0,2\pi]$}, \\
\label{iid_components4:eq}    & x_i \mbox{ is $N(0,1)$ normally distributed}.
\end{align}

\begin{corollary} \label{variance:cor}
If the components $x_i$ are i.i.d.\ with distribution \eqref{iid_components1:eq} or \eqref{iid_components4:eq}, then 
\[
\V[x^*AXx] = \frac{1}{2}\|\offdiag(A+A^T)\|_F^2,
\]
where $\| \cdot \|_F$ denotes the Frobenius norm and $\offdiag$ the offdiagonal part of a matrix. 
If the components are i.i.d.\ with distribution \eqref{iid_components2:eq} or \eqref{iid_components3:eq}, then 
\[
\V[x^*AXx] = \|\offdiag(A)\|_F^2.
\]
\end{corollary}
\begin{proof} 
For the distributions \eqref{iid_components1:eq} and \eqref{iid_components4:eq}, the components $x_i$ have only real values and $\E[x_i^2] = 1$. Therefore
\begin{eqnarray*}
\sum_{\stackrel{i,j}{i \neq j}}^n\overline{a}_{ij}a_{ij} 
+ 
\sum_{\stackrel{i,j}{i \neq j}}^n\overline{a}_{ij}a_{ji} \E[x_i^2]\E[\overline{x}_j^2] 
& = &
\sum_{\stackrel{i,j}{i \neq j}}^n\overline{a}_{ij}a_{ij} + 
\sum_{\stackrel{i,j}{i \neq j}}^n\overline{a}_{ij}a_{ji} \\
&= & 
 \frac{1}{2} \sum_{\stackrel{i,j}{i \neq j}}^n(\overline{a_{ij}+a_{ji}})(a_{ij}+a_{ji}) \\
 &=& \frac{1}{2} \|\offdiag(A+A^T)\|_F^2.
\end{eqnarray*}
For the distributions \eqref{iid_components2:eq} and \eqref{iid_components3:eq}we have
$\E[x_i^2] = 0$, and thus
\[
\sum_{\stackrel{i,j}{i \neq j}}^n\overline{a}_{ij}a_{ij} 
+ 
\sum_{\stackrel{i,j}{i \neq j}}^n\overline{a}_{ij}a_{ji} \E[x_i^2]\E[\overline{x}_j^2] 
 = 
\sum_{\stackrel{i,j}{i \neq j}}^n\overline{a}_{ij}a_{ij} 
=    \|\offdiag(A)\|_F^2.
\]
\end{proof}

In a practical situation where we approximate $\tr(A)$ by averaging over $N$ samples we can compute the root mean square deviation along 
with the averages and rely on the law of large numbers to assess the probability that the computed mean lies within the $\sigma$, $2\sigma$ 
or $3\sigma$ interval. Several results on Hutchinson's method have been formulated which go beyond this asymptotic aspects by giving tail or 
concentration bounds; see 
\cite{AvronToledo2011,CortinovisKressner2020,roosta2015improved}, e.g. For the sake of illustration we here report a 
summary of these results as given in \cite{meyer2021hutch++}. In 
our numerical examples, we will simply work with the root mean square deviation to assess accuracy.

\begin{theorem}
Let the distribution for the i.i.d.\ components of the random vectors $x^i$ be sub-Gaussian, and let $\epsilon, \delta \in (0,1)$. Then 
for $N = \mathcal{O}(\log(1/\delta)/\epsilon^2)$ we have that the probability for   
\begin{equation} \label{eq:haim_toledo}
 \left| \frac{1}{N} \sum_{i=1}^n (x^i)^*Ax^i - \tr(A) \right| \leq \epsilon \|A\|_ F
\end{equation}
is $\geq 1 - \delta$.
\end{theorem}

Note that if $A$ is symmetric positive semidefinite with $\lambda_i$ denoting its 
(non-negative) eigenvalues, then 
\[
\|A\|_F = \left(\sum_{i=1}^n \lambda_i^2 \right)^{1/2} \leq \sum_{i=1}^n \lambda_i = \tr(A),
\]
implying that \eqref{eq:haim_toledo} yields a (probabilistic) relative error bound for the trace. Also note 
that the real distributions \eqref{iid_components1:eq} and \eqref{iid_components4:eq} are sub-Gaussian; see \cite{meyer2021hutch++}.

\section{Multilevel Monte-Carlo for the trace of the inverse} \label{mlmc_trace:est}
We now turn to the situation where we want to estimate $\tr(A^{-1})$ for a large and sparse matrix $A$. Direct application of Theorem~\ref{hutchinson:thm} shows that  an unbiased estimator for $\tr(A^{-1})$ is given by
\begin{equation} \label{plain_tr_estimate:eq}
\frac{1}{N} \sum_{i=1}^N x^{(i)}A^{-1}x^{(i)} \approx \tr(A^{-1}),
\end{equation}
where the vectors $x^{(i)}$ are independent random variables satisfying \eqref{P_assumptions:eq}, and that its variance is 
\[
\frac{1}{N} \| \offdiag(A^{-1}+A^{-T})\|_F^2 \enspace \mbox { or } \enspace \frac{1}{N} \|\offdiag(A^{-1})\|_F^2,
\]
depending on whether the components of $x^{(i)}$ satisfy \eqref{iid_components1:eq}, \eqref{iid_components4:eq} or \eqref{iid_components2:eq}, \eqref{iid_components3:eq}, respectively. 

Each time we add a sample $i$ to \eqref{plain_tr_estimate:eq} we have to solve a linear system with matrix $A$ and right hand side $x^{(i)}$, and the cost for solving these linear systems determines the cost for each stochastic estimate. For a large class of matrices, multigrid methods represent particularly efficient linear solvers. We assume that this is the case for our matrix $A$ and now describe how to derive a multilevel Monte-Carlo method for the approximation of $\tr(A^{-1})$ which
uses the multigrid hierarchy not only for the linear solver, but also to obtain a good representation \eqref{mlmc_representation:eq} required for a multilevel Monte-Carlo approach.

\subsection{Derivation of a mutilevel Monte-Carlo method}
\label{sect:derivation_MLMC}
Multigrid methods rely on the interplay between a smoothing iteration and a coarse grid correction 
which are applied alternatingly. In the geometric interpretation, where we view components of vectors as representing a continuous function on a discrete grid, the smoother has the property to make the error of the current iterate smooth, i.e.\ varying slowly from one grid point to the next. Such error can be
represented accurately by a coarser grid, and the coarse grid correction solves for this coarse error
on the coarse grid using a coarse grid representation of the matrix. The solution is then interpolated back to the original ``fine'' grid and applied as a correction to the iterate. The principle can be applied recursively using a sequence of coarser grids with corresponding operators, the solves on the coarsest grid being obtained by direct factorization.

To obtain a multilevel Monte-Carlo decomposition we discard the smoother and only consider the coarse grid operators and the intergrid transfer operators which we now describe algebraically. The coarse grid operators are given by a sequence of matrices 
\[
A_\ell \in \C^{n_\ell \times n_\ell}, \ell = 1,\ldots,L,
\]
representing the original matrix $A = A_1 \in \mathC^{n_1 \times n_1}$  on the different levels $\ell = 1,\ldots,L$; the prolongation and restriction operators 
\[
P_\ell \in \C^{n_{\ell}\times n_{\ell+1}}, \; R_\ell \in \C^{n_{\ell+1} \times n_\ell}, \; \ell = 1,\ldots,L-1.
\]
transfer data between the levels.
Typically, when $A$ is Hermitian, one takes $P_\ell = R_\ell^*$, and for given $P_\ell, R_\ell$ the 
coarse system matrices $A_\ell$ are often constructed using the Petrov-Galerkin approach
\[
A_{\ell+1} = R_\ell A_\ell P_\ell, \; \ell = 1,\ldots, L-1.
\]
Using the accumulated prolongation and restriction operators
\[
\hat{P}_\ell =  P_1 \cdots P_{\ell-1}  \in \C^{n \times n_\ell}, \hat{R}_\ell = R_{\ell-1} \cdots R_1 \in \C^{n_\ell \times n}, \; \ell = 1,\ldots, L,
\]
where we put $\hat{R}_1 = \hat{P}_1 = I \in \C^{n \times n}$ by convention, we regard $\hat{P}_\ell A_\ell^{-1} 
\hat{R}_\ell$ as the approximation to $A^{-1}$ at level $\ell$. We thus obtain a 
multilevel decomposition for the trace as
\begin{equation} \label{tr_ml_dec:eq}
\tr(A^{-1}) = \sum_{\ell=1}^{L-1} \tr\left(\hat{P}_\ell A_\ell^{-1} \hat{R}_\ell- \hat{P}_{\ell+1} A_{\ell+1}^{-1} \hat{R}_{\ell+1}\right) + \tr(\hat{P}_L A_L^{-1} \hat{R}_L).
\end{equation}
This gives
\[
\tr(A^{-1}) = \sum_{\ell=1}^{L-1}  \E\left[(x^\ell)^*\left(\hat{P}_\ell A_\ell^{-1} \hat{R}_\ell - \hat{P}_{\ell+1} A_{\ell+1}^{-1} \hat{R}_{\ell+1} \right)x^\ell\right] + \E\left[(x^L)^*\hat{P}_L A_L^{-1} \hat{R}_Lx^L\right],
\]
with the components of $x^\ell \in \mathC^n$ being i.i.d.\ stochastic variables satisfying \eqref{P_assumptions:eq}.
The unbiased multilevel Monte-Carlo estimator is then
\begin{eqnarray*}
\tr(A^{-1}) &\approx& \sum_{\ell=1}^{L-1} \sum_{i=1}^{N_\ell} \left( (x^{i,\ell})^*\hat{P}_\ell A_\ell^{-1} \hat{R}_\ell x^{i,\ell}- (x^{i,\ell})^*\hat{P}_{\ell+1} A_{\ell+1}^{-1} \hat{R}_{\ell+1}x^{i,\ell}\right) \\
& & \mbox{}+ \sum_{i=1}^{N_L}(x^{n,L})^*\hat{P}_L A_L^{-1} \hat{R}_Lx^{i,L},
\end{eqnarray*}
where the vectors $x^{i,\ell} \in \C^n$ are stochastically independent samples of the random variable $x \in \mathC^n$ satisfying \eqref{P_assumptions:eq}.

The following remarks collect some important observations about this stochastic estimator.

\begin{remark} 
Computationally, the estimator requires to solve systems of the form $A_\ell y^{n,\ell} = z$ with $z = \hat{R}_{\ell}x^{n,\ell} $. Since the matrices $A_\ell$ arise from the multigrid hierarchy, we directly have a multigrid method available for these systems by restricting the method for $A$ to the levels $\ell, \ldots,L$.
\end{remark}
\begin{remark}
Since for any two matrices $B = (b_{ij}) \in \C^{n \times m}$ and $C = (c_{kl}) \in \C^{m \times n}$ the trace of their product does not depend on the order, 
\begin{equation} \label{trace_commute:eq}
 \tr(BC) = \sum_{i=1}^n \sum_{j=1}^m b_{ij}c_{ji} = \sum_{j=1}^m \sum_{i=1}^n c_{ji}b_{ij} = \tr(CB),   
\end{equation} 
we have
    \[
       \tr(\hat{P}_LA^{-1}_L\hat{R}_L) = \tr(A_L^{-1}\hat{P}_L\hat{R}_L).
     \]
     So, instead of estimating the contribution $\tr(\hat{P}_LA_L^{-1}\hat{R}_L)$ in \eqref{tr_ml_dec:eq} stochastically, we can also compute it directly by inverting the matrix $A_L$ and computing the product $A_L^{-1}\hat{R}_L\hat{P}_L$. Note that $\hat{R}_L$ and $\hat{P}_L$ are 
     usually sparse with a maximum of $d$, say, non-zero entries per row. The arithmetic work for $A_L^{-1}\hat{R}_L\hat{P}_L$ is thus of order $\mathcal{O}(dN_L^2)$ for the product $\hat{R}_L\hat{P}_L$ plus $\mathcal{O}(n_L^3)$ for the inversion of $A_L$ and the product $A_L^{-1}(\hat{R}_L\hat{P}_L$). Since the variance of $x^*\hat{P}_LA^{-1}_L\hat{R}_Lx$ is presumably large, this direct computation can be much more efficient than a stochastic estimation, even when we aim at only quite low precision in the stochastic estimate.
\end{remark}
\begin{remark} \label{rem:simplified}
     There are situations where $\hat{R}_\ell \hat{P}_\ell = I \in \C^{n_\ell \times n_\ell}$, for example in aggregation based multigrid methods, where the columns of $P_\ell$ are orthonormal and $R_\ell = P_\ell^*$, see \cite{Braess95, Brezinaetal2005}. Then
     \begin{eqnarray*}
     \tr(\hat{P}_\ell A_\ell^{-1}\hat{R}_\ell) &=& \tr(A_\ell^{-1}\hat{R}_\ell\hat{P}_\ell) \, = \,  \tr(A_\ell^{-1}),
     \end{eqnarray*}
and
  \begin{eqnarray*}
      \tr(\hat{P}_{\ell+1} A_{\ell+1}^{-1}\hat{R}_{\ell+1}) &=& \tr(\hat{P}_{\ell}P_{\ell}A_{\ell+1}^{-1}R_{\ell} \hat{R}_\ell) \, 
      = \, \tr(P_{\ell}A_{\ell+1}^{-1}R_{\ell} \hat{R}_\ell\hat{P}_{\ell}) 
      \, = \, \tr(P_{\ell}A_{\ell+1}^{-1}R_{\ell}).
     \end{eqnarray*}
     This means that instead of the multilevel decomposition \eqref{tr_ml_dec:eq} we can use
     \[
        \tr(A) = \sum_{\ell=1}^{L-1}\left( \tr(A_\ell^{-1})- \tr({P}_{\ell} A_{\ell+1}^{-1} {R}_{\ell})\right) + \tr(A_L^{-1}),   
     \]
     in which the stochastic estimation on level $\ell$ now involves random vectors from $\C^{n_\ell}$ instead of $\C^n$. 
\end{remark}

\subsection{Discussion of the multilevel Monte-Carlo method}
A profound analysis of the proposed multilevel Monte-Carlo method must take the approximation properties
of the representation of the matrix at the various levels into account. This is highly problem 
dependent, so that in this paper we only provide a discussion of heuristics on why the proposed approach has the potential to yield efficient multilevel Monte-Carlo schemes. 

To simplify the discussion to follow, let us assume that the variance  of the estimator at level $\ell$ is given by the square of the Frobenius norm of the off-diagonal part.
This is the case, for example, if the components are i.i.d.\ with distribution \eqref{iid_components2:eq} or \eqref{iid_components3:eq}; see Corollary~\ref{variance:cor}. This Frobenius norm can be related to the singular values of $A$. Recall that the singular value decomposition of a non-singular matrix $A$ is
\begin{eqnarray}
A &=& U \Sigma V^* \enspace \mbox { with }  U,\Sigma,V \in \C^{n\times n}, U^*U = V^*V = I, \label{svd:eq}\\
 & & U = [u_1 | \cdots | u_n], \; V= [v_1| \ldots |v_n], \nonumber \\
 & & \Sigma = \diag(\sigma_1,\ldots,\sigma_n), 0< \sigma_1 \leq \cdots \leq \sigma_n, \nonumber
\end{eqnarray}
with left singular vectors $u_i$, right singular vectors $v_i$ and positive singular values $\sigma_i$ which we ordered by increasing value for convenience here. 
In the following we base all our discussion on singular values and vectors. It is therefore worthwhile to mention that in the case of a Hermitian matrix $A$ this discussion simplifies in the sense that then the singular values are the moduli of the eigenvalues, and left and right singular vectors are identical and coincide with the eigenvectors.

\begin{lemma} \label{offdiagnorm:lem} Let $A \in \C^{n \times n}$ have singular values $\sigma_i, i=1,\ldots,n$. Then 
\begin{equation} \label{svd_estimate:eq}
\| \offdiag(A) \|^2_F = \sum_{i=1}^n \sigma_i^2 - \sum_{i=^1}^n |a_{ii}|^2.
\end{equation}
\end{lemma}
\begin{proof} The equality $\|A\|^2_F = \sum_{i=1}^n \sigma_i^2$ is a basic fact from linear algebra, see \cite{GvL}, e.g. The formula \eqref{svd_estimate:eq} uses this and corrects for the vanishing diagonal part in $\offdiag(A)$.
\end{proof}

For the trace of the inverse $A^{-1}$ we have
\begin{equation} \label{sigular_values_offdiag:eq}
\| \offdiag(A^{-1}) \|^2_F = \sum_{i=1}^n \sigma_i^{-2} - \sum_{i=^1}^n |(A^{-1})_{ii}|^2.
\end{equation}
since the reciprocals of the singular values of $A$, are the singular values of $A^{-1}$. Therefore, in a simplified manner---disregarding the second term in \eqref{sigular_values_offdiag:eq}---it appears that the small singular values of $A$ are  those who contribute most to the variance for the Hutchinson estimator \eqref{plain_tr_estimate:eq} for  $\tr(A^{-1})$. In high performance computing practice, {\em deflation} has thus become a common tool, see \cite{DeGrand:2004qw,Endress:2014qpa,Gambhir_2017,Giusti:2004yp}, e.g., to reduce the variance: 
One precomputes the $k$, say, smallest singular values $\sigma_1,\ldots,\sigma_k$ of 
$A$ in the singular value decomposition  \eqref{svd:eq} together with their left singular vectors $u_1,\ldots,u_k$. With the orthogonal 
projector
\begin{equation} \label{eq:orth_proj}
\Pi = U_k U_k^*, \mbox{ where } U_k = [u_1 | \cdots | u_k],
\end{equation}
we now have $A^{-1} = A^{-1} (I-\Pi) + A^{-1}\Pi$ with
\begin{equation} \label{eq:defl}
 A^{-1} (I-\Pi)  = \sum_{i=k+1}^n v_i \sigma_i^{-1} u_i^*, \quad A^{-1} \Pi  = \sum_{i=1}^k A^{-1}u_iu_i^*.
\end{equation}
This shows that in $A^{-1}(I-\Pi)$ we have deflated the small singular values of $A$, so that we can expect a reduction of the variance when estimating the trace of this part stochastically. The trace of the second part  is equal to 
the sum $\sum_{i=1}^k u_i^* A^{-1}u_i$ (see \eqref{trace_commute:eq}), and $A^{-1}u_i = \sigma_i^{-1}v_i$. So the second part can be computed directly from the singular triplets computed for the deflation. 
If $A$ is Hermitian, the deflation approach simplifies and amounts to precomputing the $k$ smallest eigenpairs. We refer to the results in \cite{Gambhir_2017} for a more in-depth analysis and discussion about the heuristics just presented.

The deflation approach is still quite costly, since one has to precompute the singular values and vectors, and if the size of the matrix increases it is likely that we have to increase $k$ to maintain the same reduction in the variance. Approximate deflation has thus been put forward as an alternative, see \cite{Balietal2015,Romero_2020}, where one can use larger values for $k$ while at the same time allowing that the contribution of the small singular values to the variance is eliminated only approximately. One then replaces $\Pi$ by a more general projector of the form
\[
\Pi = \hat{U}_k(\hat{V}_k^*A\hat{U}_k)^{-1}\hat{V}_k^*A, \enspace \hat{U}_k, \hat{V}_k \in \C^{n \times k}
\]
where now $\hat{U}_k$ and $\hat{V}_k$ can be regarded as containing approximate left and right singular vectors, respectively, as their columns. Actually, it is sufficient that their range is spanned by such approximations to left and right singular vectors, since the construction of $\Pi$ is invariant under transformations $\hat{U} \to \hat{U}B_U, \hat{V} \to \hat{V}B_V$ with non-singular matrices $B_U,B_V \in \C^{k \times k}$. In the decomposition $A^{-1} = A^{-1}(I-\Pi) + A^{-1}\Pi$ we now have, again using \eqref{trace_commute:eq},
\begin{eqnarray*}
\tr(A^{-1}(I-\Pi)) &=&  \tr(A^{-1}) -  \tr(\hat{U}_k(\hat{V}_k^*A\hat{U}_k)^{-1}\hat{V}_k^*) , \\
\tr(A^{-1}\Pi) &=& \tr( \hat{U}_k(\hat{V}_k^*A\hat{U}_k)^{-1}\hat{V}_k^*) .
\end{eqnarray*}
 If $k$ is relatively small, the second trace can be computed directly as in the exact deflation approach. If we take larger values for $k$, we can estimate it stochastically. The inexact deflation approach then becomes a two-level Monte-Carlo method.
 
 If we look at our multilevel Monte-Carlo decomposition \eqref{multilevel_decomposition:eq} with just tow levels, then it differs from inexact deflation in that the value for $k$ is now 
 very large, namely the grid size at level 2 which usually is $\mathcal{O}(n)$. The 
 matrix $\hat{U}_k$ spanning the approximate singular vectors is replaced by the prolongation operator $P_1$, and $\hat{V}_k^*$ corresponds to the restriction operator $R_1$. The multigrid construction principle should ensure that the range of $P_1$ contains good approximations to $\mathcal{O}(n)$ left singular vectors belonging to small singular values, and similarly for $R_1^*$ with respect to right singular vectors. This is why the variance reduction can be expected to be efficient. We thus have a large value of $k$---proportional to $n$---which targets at a high reduction of the variance of the first term. The second term involves the second level matrix representation, which is still of large size, and its trace estimator will, in addition, still have large variance. This is the reason why we extend the approach to involve many levels, ideally until a level $L$ where we can compute the trace directly, so that we do not suffer from a potentially high variance of a stochastic estimator. 
 
 To conclude this discussion, we note that several other techniques for variance reduction have been suggested which can also be regarded as two-level Monte-Carlo techniques. For example, \cite{liu2014polynomial,baral2019} take a decomposition $A^{-1}-p(A) + p(A)$ with an appropriately chosen polynomial $p(A)$ and then estimates $\tr(A^{-1}-p(A))$ stochastically. The ``truncated solver'' method of \cite{Alexandrou_2014} follows a related idea by subtracting an approximation to the inverse. A similar decomposition with $p$ being a truncated Chebyshev series approximation was considered in \cite{Hanetal2017,Han2015,SaadUbaru2017}, for example, in which case $\tr(A^{-1}-p(A)$ is actually neglected. The work then resides in the stochastic estimation of $\tr(p(A))$, thus avoiding to solve linear systems.  
 
 Finally, we refer to \cite{meyer2021hutch++} for a recent further variance reduction technique for Hutchinson's method, enhancing it by using vectors of the form $Av$ with random vectors $v$.

\section{Numerical results} \label{numerical_results:sec}
We consider three classes of matrices: The standard discrete 2d Laplacian, the 2d gauge Laplacian, and the Schwinger model. These three classes represent an increasingly complex hierarchy of problems which will eventually lead to our final, though yet unreached target, the Wilson-Dirac matrix arising in lattice QCD. 
The improvements of 
the multilevel approach compared to ``plain'' Hutchinson 
\eqref{plain_tr_estimate:eq} are tremendous and typically reach 
two orders of magnitude or more. This is why we compare against {\em 
deflated Hutchinson}, where we deflate the $n_\defl$ smallest 
eigenpairs of the matrix $A$. With $U \in \mathC^{n \times n_\defl}$ 
holding the  respective eigenvectors in its columns, we use the 
projector $\Pi = I-UU^*$ as in \eqref{eq:orth_proj}, resulting in the
decomposition \eqref{eq:defl} Therein we estimate $\tr(A^{-1}(I-\Pi))$ 
with the Hutchinson estimator whereas $\tr(A^{-1}\Pi) = \sum_{i=1}^{n_\defl} \lambda_i^{-1}$ is obtained directly from the deflated eigenpairs. We always performed a rough scan to 
determine a number $n_\defl$ of deflated 
eigenpairs which s close to time-optimal. The deflated Hutchinson approach usually gains at least one order of magnitude in time and arithmetic cost over plain Hutchinson. 

All our computations were done on a single thread of an Intel Xeon Processor E5-2699 v4, with a MATLAB R2021a implementation of our numerical experiments for the 2d Laplacian, and in Python for our tests with the gauge Laplacian and the Schwinger model. By default we aimed at a relative accuracy of $\epsilon = 10^{-3}$. This is done as follows: We first perform five stochastic estimates, take their mean and subtract their root mean square deviation, giving the value $\tau$. In the deflated Hutchinson method we now perform stochastic estimates as long as their root mean square deviation exceeds $\epsilon \tau$. For the multilevel Monte-Carlo method we have to prescribe a value for the root mean square deviations $\rho_\ell$ for the stochastic estimation of
each of the traces 
\begin{equation} \label{eq:trace_diff}
 \tr\left(\hat{P}_\ell A_\ell^{-1} \hat{R}_\ell- \hat{P}_{\ell+1} A_{\ell+1}^{-1} \hat{R}_{\ell+1}\right), \enspace \ell = 1,\ldots, L-1
\end{equation}
from \eqref{tr_ml_dec:eq}, while we always compute the last term $\tr(\hat{P}_L A_L^{-1} \hat{R}_L)$ in \eqref{tr_ml_dec:eq} non-sto\-cha\-stic\-al\-ly as $\tr(A_L^{-1} \hat{R}_L\hat{P}_L)$, inverting $A_L$ explicitly. The requirement is to have
\[
\sum_{\ell=1}^{L-1} \rho_\ell^2 = (\epsilon \tau)^2,
\]
so the obvious approach is to put $
\rho_\ell = \epsilon\tau/\sqrt{L-1}$ for all $\ell$, and this is what we do in our first two examples. It might be advantageous, though, to allow for a larger value of $\rho_\ell$ on those level differences where the cost is high, and we do so in Example~\ref{ex:Schwinger}. 

For each stochastic estimate for \eqref{eq:trace_diff} we have to solve linear systems with the matrices $A_\ell$ and $A_{\ell+1}$. This is done using 
a multigrid method based on the same prolongations $P_\ell$, restrictions $R_\ell$ and coarse grid operators $A_\ell$ that we use to obtain our multilevel decomposition \eqref{tr_ml_dec:eq}.
However, when multigrid is used as a solver, we use the full hierarchy going down to coarse grids of very small sizes, whereas in the multilevel decomposition \eqref{tr_ml_dec:eq} used in multilevel Monte-Carlo we might stop at an earlier level.

For all experiments we report mainly two quantities. The first is the number of stochastic estimates that are performed at each level difference \eqref{eq:trace_diff} for multilevel Monte-Carlo together with the number of stochastic estimates in deflated Hutchinson (which always require linear solves at the finest level). These numbers may be interpreted as illustrating how multilevel Monte-Carlo moves the higher variances to the coarser level differences. As a second quantity, we report the approximate arithmetic cost for both methods, deflated Hutchinson and multilevel Monte-Carlo, which we obtain using the following cost model: For every matrix-vector product 
of the form $Bx$ we assume a cost of $\nnz(B)$, the number of nonzeros in $B$. In this manner ,one unit in the cost model roughly corresponds to a multiplication plus an addition. This applies to the computation of residuals, of prolongations and restrictions and the coarsest grid solve in the multigrid solver as well as to the ``global'' restrictions and prolongations $\hat{R}_\ell, \hat{P}_\ell$ used in each stochastic estimate in multilevel Monte-Carlo. For the latter method, we also count the work for the direct computation of the trace at the coarsest level, which involves the inversion of the coarsest grid matrix and additional matrix-matrix products. This cost model thus only neglects vector-vector and scalar operations and is thus considered sufficiently accurate for our purposes.  

\begin{example} \label{ex:2dLaplace} The discrete 2d Laplacian is the $N^2 \times N^2$-matrix
\[
L^N = B \otimes I + I \otimes B, \enspace \text{$I$ the $N \times N$ identity,  }	B_0 = \begin{bmatrix}
				2 & -1\\
				-1 & 2 & \ddots\\
				& \ddots & \ddots & -1 \\
				& & -1 & 2
			\end{bmatrix} \in \mathC^{N \times N},
\]
which results from the finite difference approximation of the Laplace operator on an equidistant grid in the unit square with Dirichlet boundary conditions. Note that the eigenvalues of $L^N$ are explicitly known, so the trace of the inverse is, in principle, directly available as the sum of the inverses of the eigenvalues. 

\begin{table}
    \centering
    \begin{small}
    \begin{tabular}{|r|l|cccccc|cc|}
    \hline
    \multicolumn{10}{|c|}{2d Laplace} \\
      \hline
     $N$  &                &$\ell = 1$&$\ell = 2$&$\ell = 3$&$\ell = 4$&$\ell = 5$&$\ell = 6$& $n_\defl$ & $L$ \\ \hline 
     63   & $n_\ell$       & $63^2$   & $31^2$   & $15^2$   &          &          &          &  92       & 3 \\
          & $\nnz(L^N_\ell)$ & $1.96e4$ & $8.28e3$ & $1.85e3$ &          &          &          &           &  \\ \hline 
    127   & $n_\ell$       & $127^2$  & $63^2$   & $31^2$   &  $15^2$  &          &          &  44       & 4 \\
          & $\nnz(L^N_\ell)$ & $8.01e4$ & $3.50e4$ & $8.28e3$ & $1.85e3$ &          &          &           &  \\ \hline     
    255   & $n_\ell$       & $255^2$  & $127^2$  & $63^2$   & $31^2$   &  $15^2$  &          &  64       & 5 \\
          & $\nnz(L^N_\ell)$ & $3.24e5$ & $1.44e5$ & $3.50e4$ & $8.28e3$ & $1.85e3$ &          &           &  \\ \hline            
    511   & $n_\ell$       & $511^2$  & $255^2$  & $127^2$  & $63^2$   & $31^2$   &  $15^2$  &  76       & 6 \\
          & $\nnz(L^N_\ell)$ & $1.30e6$ & $5.82e5$ & $1.44e5$ & $3.50e4$ & $8.28e3$ & $1.85e3$ &           &     \\ \hline               
    \end{tabular}
    \end{small}
        \caption{Parameters and quantities for Example~\ref{ex:2dLaplace}}
    \label{tab:2dLaplace}
\end{table}

For the multigrid hierarchy we choose $P_\ell$ to be the standard bilinear interpolation from a grid of size $N_{\ell+1} \times N_{\ell +1}$ to one of size $N_\ell \times N_\ell$; see \cite{Trottenberg2000}. Here, the 
number $N_\ell$ of grid points in one dimension on level $\ell$ is recursively given as $N_{\ell+1} = \lfloor  N_{\ell}/2 \rfloor $. The restrictions $P_\ell$ are taken as the adjoints of the interpolations $P_\ell$, and the coarse grid operators $L^N_\ell$ are obtained as Galerkin approximations $L^N_{\ell+1} = R_\ell L^N_{\ell}P_\ell$. 
So the operator $L^N_\ell$ at level $\ell$ is an $n_\ell \times n_\ell$ matrix with $n_\ell = N_\ell^2$.

V-cycle multigrid with one step of Gauss-Seidel pre- and post-smoothing  was used as a solver to compute $(L^N_\ell)^{-1}x$ on the various levels. In the solver, we always use the full hierarchy down to a smallest coarsest grid operator size of $N^2 = 7^2$, where we inverted directly using a Cholesky factorization.   For multilevel Monte-Carlo we took $L$, the maximum number of levels, such that $N_L = 15$, since it turned out that with this choice the work for the direct computation of the trace at level $L$---requiring the inversion of a matrix of size $15^2 = 225$---was small enough compared to the other cost. Table~\ref{tab:2dLaplace} summarizes the most important quantities for the matrices in this example.

\begin{figure}[htbp]
    \includegraphics[width=.49\textwidth]{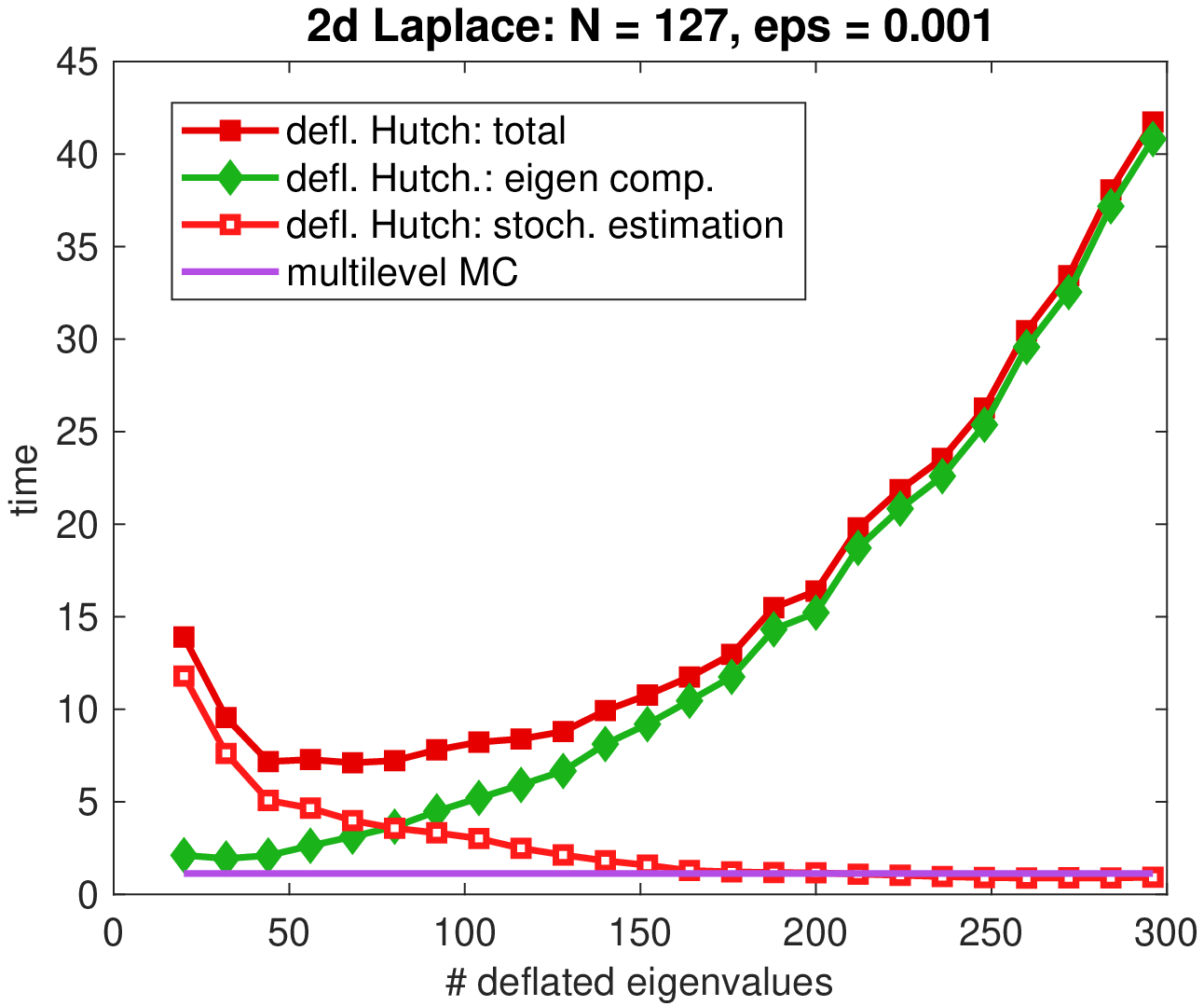} \hfill
    \includegraphics[width=.49\textwidth]{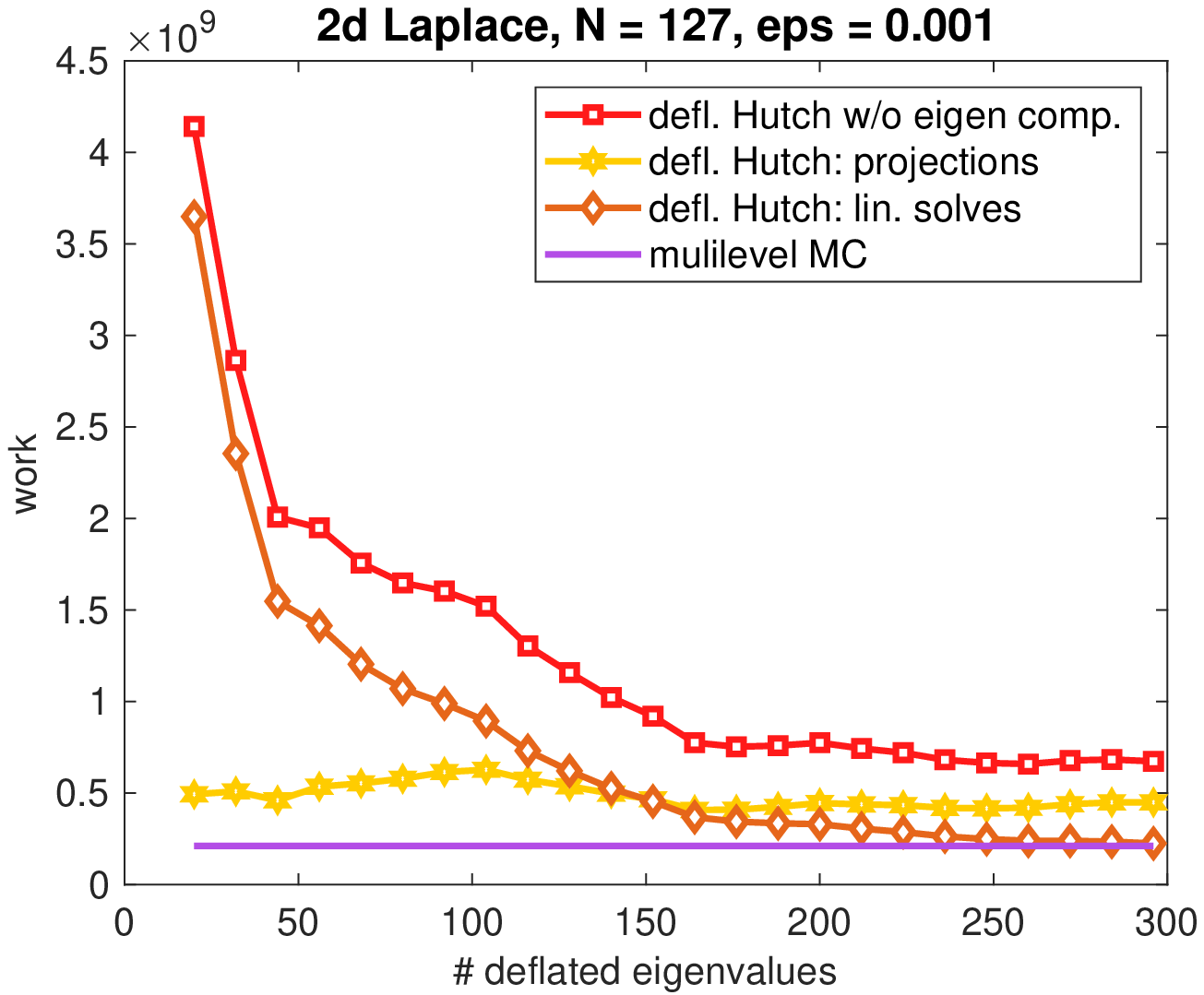}
  \centering
  \caption{2d Laplace: Comparison of execution times and arithmetic cost for multilevel Monte-Carlo and deflated Hutchinson, varying $n_\defl$. \label{fig:2dLaplace_timings}}
\end{figure}

Figure~\ref{fig:2dLaplace_timings} reports timings and arithmetic cost for a study to approximately find the time-optimal number of eigenvalues to deflate in the deflated Hutchinson method for $N=127$. The left part of the figure shows that there is a substantial decrease in the timings when we start to increase the number of deflated eigenvalues 
up to around 50, after which the cost for computing the eigenpairs becomes increasingly larger, eventually dominating the overall execution time completely. The straight horizontal line on the bottom indicates the time required in the multilevel Monte-Carlo approach using $L=4$ levels. Even with an optimal number of deflated eigenvalues, deflated Hutchinson takes about 7 times longer than multilevel Monte-Carlo. The right part of Figure~\ref{fig:2dLaplace_timings} reports the arithmetic cost. We do {\em not} include the cost for the computation of the eigenpairs, since we ignore the details of Matlab's \texttt{eigs} function that we use to compute those. The top line in the right part of the figure thus corresponds to the line marked with open squares in the left part, and we see that the cost model quite accurately matches the observed execution times. The two lines below the top line separate this work into the work spent in the computation of the projections $(I-UU^*)x = x - U(U^*x)$, which is $2nn_\defl$, and the work spent in the linear solves. The straight horizontal line represents the work for multilevel Monte-Carlo. We see that if a large number of eigenpairs is deflated, we have comparable work in the linear solves in deflated Hutchinson than in multilevel Monte-Carlo, due to the fact that we have to do a similar number of stochastic estimates. Each estimate, though, now becomes substantially more expensive due to the additional projections, so that we still see a factor of about 3 in favor of multilevel Monte-Carlo. We repeat that this comparison does not take the work for computing the eigenpairs in deflated Hutchinson into account.

\begin{figure}[htbp]
\includegraphics[width=.49\textwidth]{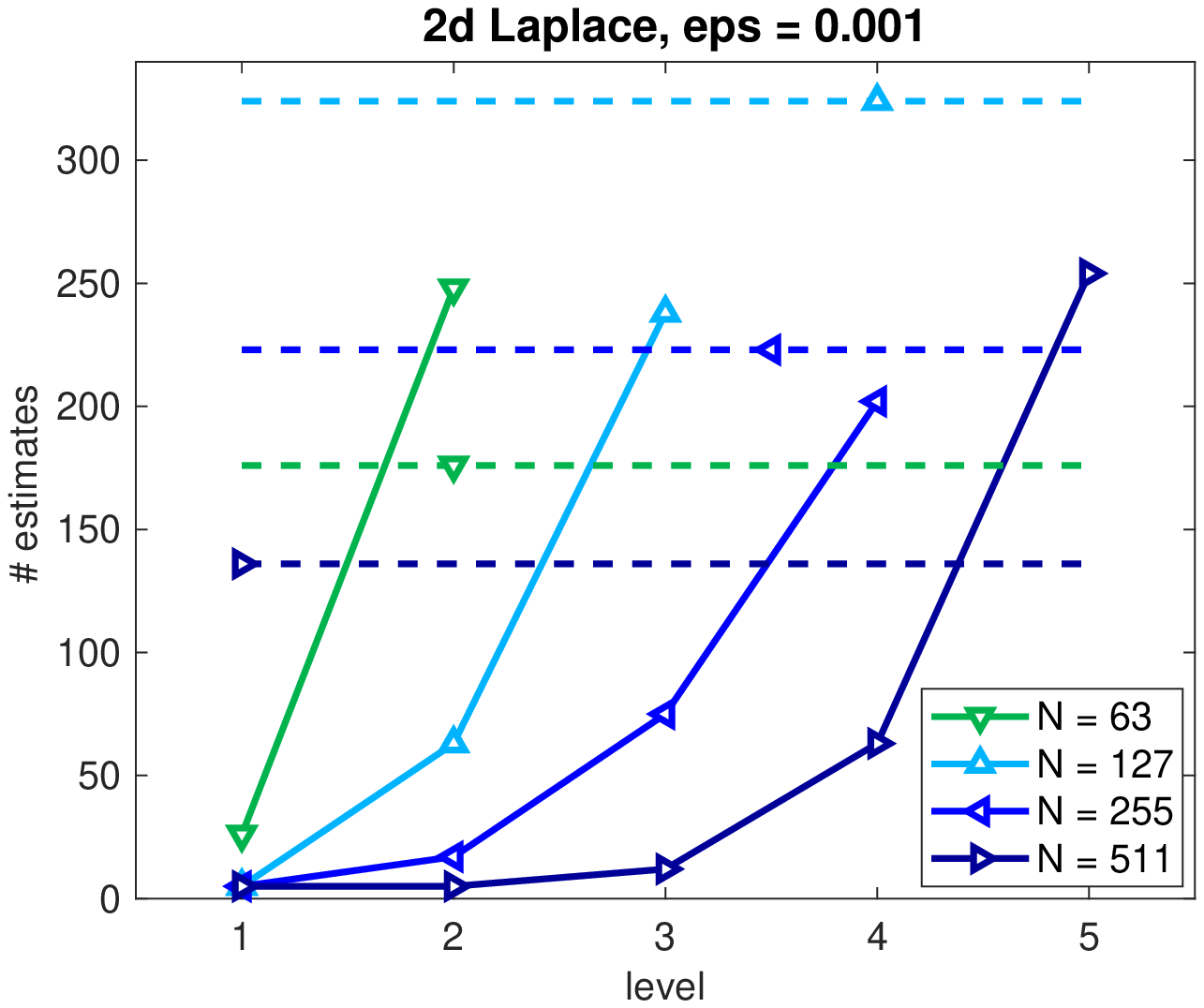} \hfill
\includegraphics[width=.49\textwidth]{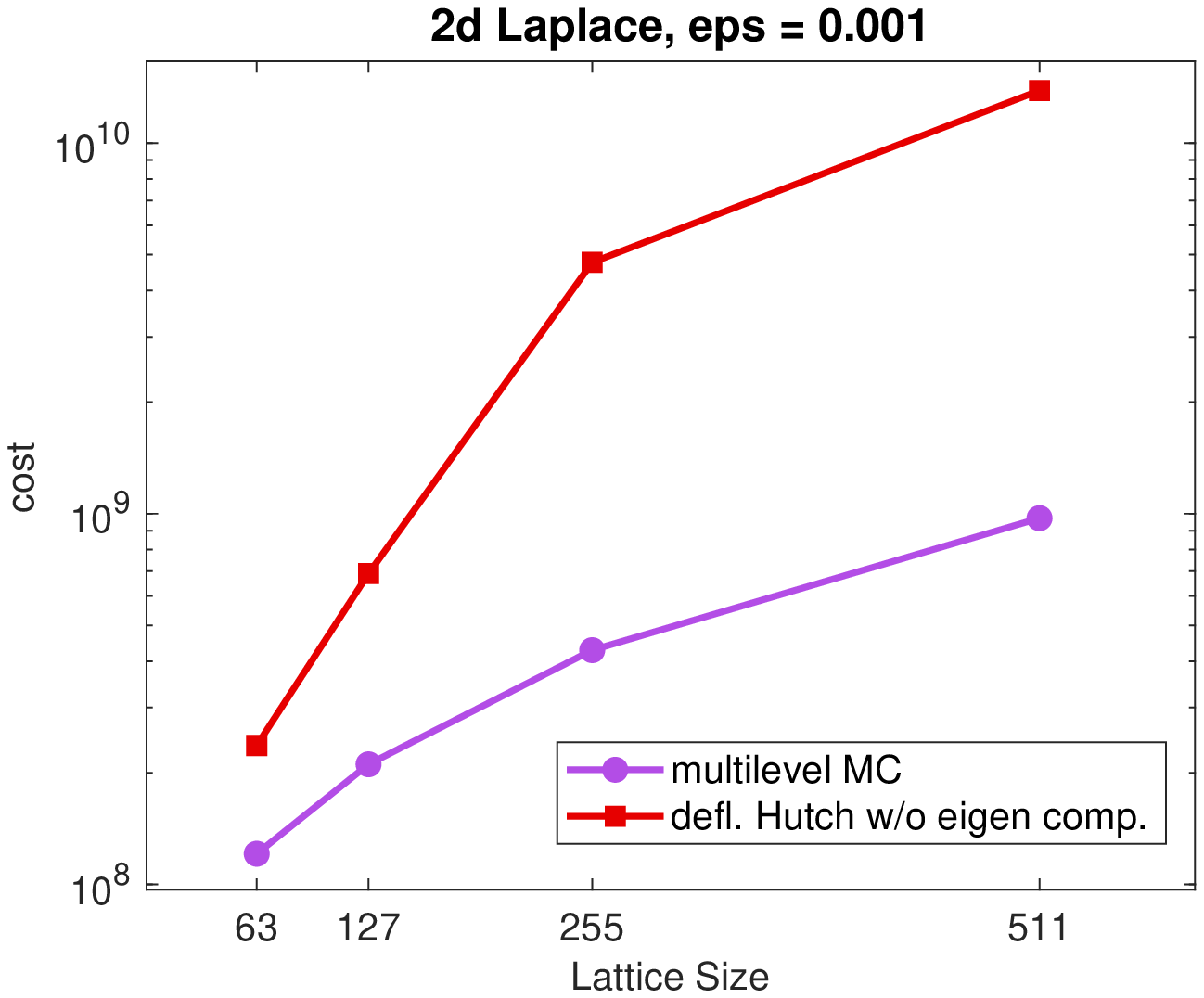}
  \caption{Comparison of multilevel Monte-Carlo and optimally deflated Hutchinson for the 2d Laplace matrix: no of stochastic estimates on each level difference \eqref{eq:trace_diff} and total work for different $N$. \label{fig:2dLaplace_estimates_work}}
\end{figure}

We proceed with Figure~\ref{fig:2dLaplace_estimates_work} which reports 
results for the four size parameters $N=63,127,255,511$. For multilevel Monte-Carlo, the left part of the figure shows the number of stochastic estimates that we perform for each of the differences $\hat{P}_\ell (L^N_\ell)^{-1}\hat{R}_\ell - \hat{P}_{\ell+1}(L^N_{\ell+1})^{-1}\hat{R}_{\ell+1}$ in \eqref{eq:trace_diff}. 
For comparison, the number of stochastic estimates required in (time optimally) deflated Hutchinson are indicated as dashed horizontal lines; see Table~\ref{tab:2dLaplace} for the values of $n_\defl$ that we used. 
The right part of Figure~\ref{fig:2dLaplace_estimates_work} shows the total amount of work invested in the different methods where, again, we do not count the work for the eigenpair computation in deflated Hutchinson. The figure illustrates the fact that in multilevel Monte-Carlo we do few estimates on the fine, expensive levels and more on the coarse and cheap levels. The plots for the cost in the right part show how this translates into a reduction
of the arithmetic cost, reaching about one order of magnitude for the larger values of $N$, even without accounting for the cost for the computation of the eigenpairs in deflated Hutchinson.

\begin{figure}[htbp]
\centerline{\includegraphics[width=.49\textwidth]{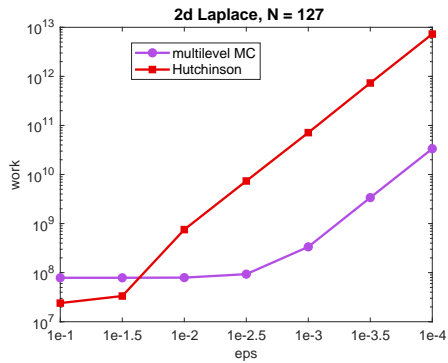}} 
  \caption{Arithmetic cost for multilevel Monte-Carlo and standard Hutchinson as a function of the target accuracy $\epsilon$. \label{fig:2dN127scaneps}}
\end{figure}

Our multilevel Monte-Carlo approach performs a Monte-Carlo estimate for each level difference. 
It therefore exhibits the same quadratic scaling with respect to the target accuracy $\epsilon$ 
as standard Hutchinson does. This is illustrated in Figure~\ref{fig:2dN127scaneps} for $N=127$. 
This is a logscale plot, and the quadratic dependence on $\epsilon^{-1}$ is clearly visible for 
the smaller values of $\epsilon$. For the larger values of $\epsilon$ (to the left), the 
results are less significant, since we always do a minimum of at least 5 stochastic estimates for each level difference. For $\epsilon = 10^{-1}$ and $\epsilon = 10^{-1.5}$, this is precisely what happens for all the level differences, so the cost for multilevel Monte-Carlo is exactly the same (and probably unnecessarily high) for both these values of $\epsilon$.  Note that the 
arithmetic cost of standard Hutchinson is between 2 and 3 orders of magnitude higher than that of multilevel Monte-Carlo for $\epsilon \leq 10^{-2}$.
\end{example}

\begin{example} \label{ex:gaugeLaplace} The gauge Laplacian $G^N$ is a modification of the standard 2d discrete Laplace matrix where now the coupling coefficients---called gauge links---are complex numbers of modulus one but with a random phase. The gauge Laplacian represents a first step towards the modeling of gauge field theories in physics, where the random coefficients model the fluctuating background gauge field. With $u_{ij}$ describing the variables belonging to a grid point $(ih,jh)$, $h=1/N$, $i,j = 0,\ldots,N-1$, the coupling described by one row of the gauge Laplacian $G^N \in \mathC^{N^2 \times N^2}$ is given as
\begin{eqnarray*}
&4 u_{ij} - e^{\i\Theta_{ij}}u_{i+1,j}-e^{\i\Phi_{ij}}u_{i,j+1} - e^{-\i\Theta_{i-1,j}}u_{i-1,j}-e^{-\i\Phi_{i,j-1}}u_{i,j-1},& \\  &i,j=0,\ldots,N-1,&
\end{eqnarray*}
where the indices are to be understood $\bmod N$ since we have periodic boundary conditions. Gauge Laplacians are Hermitian and positive semidefinite. Typically, as in our case, they are even positive definite.

\begin{table}
    \centering
    \begin{small}
    \begin{tabular}{|r|l|cccc|cc|}
    \hline
    \multicolumn{8}{|c|}{Gauge Laplace} \\
      \hline
     $N$  &                &$\ell = 1$&$\ell = 2$&$\ell = 3$&$\ell = 4$& $n_\defl$ & $L$ \\ \hline 
     64   & $n_\ell$       & $4096$   & $1354$   & $134$   &          &   60       & 3 \\
          & $\nnz(G^N_\ell)$ & $20480$ & $24900$ & $3172$ &          &          &           \\ \hline 
    128   & $n_\ell$       & $16384$  & $5440$   & $554$     &         &            60       & 3 \\
          & $\nnz(G^N_\ell)$ & $81920$ & $99448$ & $11300$ &        &          &           \\ \hline     
    256   & $n_\ell$       & $65536$  & $21802$  & $2348$   & $196$   &           20       & 4 \\
          & $\nnz(G^N_\ell)$ & $327680$ & $394628$ & $49416$ & $6352$ &          &            \\ \hline            
    512   & $n_\ell$       & $262144$  & $87296$  & $9562$  & $924$   &        20       & 4 \\
          & $\nnz(G^N_\ell)$ & $327680$ & $394628$ & $49416$ & $6352$ &          &            \\ \hline           
    \end{tabular}
    \end{small}
        \caption{Parameters and quantities for Example~\ref{ex:gaugeLaplace}}
    \label{tab:GaugeLaplace}
\end{table}

The Python package pyAMG, see \cite{OlSc2018}, contains functions to produce gauge Laplacians with prescribed distributions for the phases $\Theta_{ij}$ and $\Phi_{ij}$ of the gauge links. For our experiments, we produced gauge Laplacians with grid spacing $a=1$ and temperature $\beta = 0.009$ for the distribution of the phases of the gauge links. 

The pyAMG package provides a variety of algebraic multigrid methods. To build a multigrid hierarchy for gauge Laplacians, we use adaptive smoothed aggregation, where for the adaptive setup we took the parameters 
\[
\text{\texttt{num\_candidates} } =2, \text{\texttt{ candidate\_iters} } = 5, 
\text{\texttt{ improvement\_iters} } =8.
\]
As our linear solver we used V-cycle multigrid with one step of Gauss-Seidel pre- and post-smoothing. Table~\ref{tab:GaugeLaplace} summarizes 
the most important quantities for the same four lattice sizes as we had for the 2d Laplacian. The table shows that smoothed aggregation yields matrices at level 2 which are a factor 3 less in size but are significantly more dense, since they have even more non-zeros than the matrices at level 1. For the subsequent levels the coarsening is quite aggressive, reducing the sizes by factors of roughly 10, and similarly for the number of non-zeros. 

\begin{figure}[htbp]
\includegraphics[width=.49\textwidth]{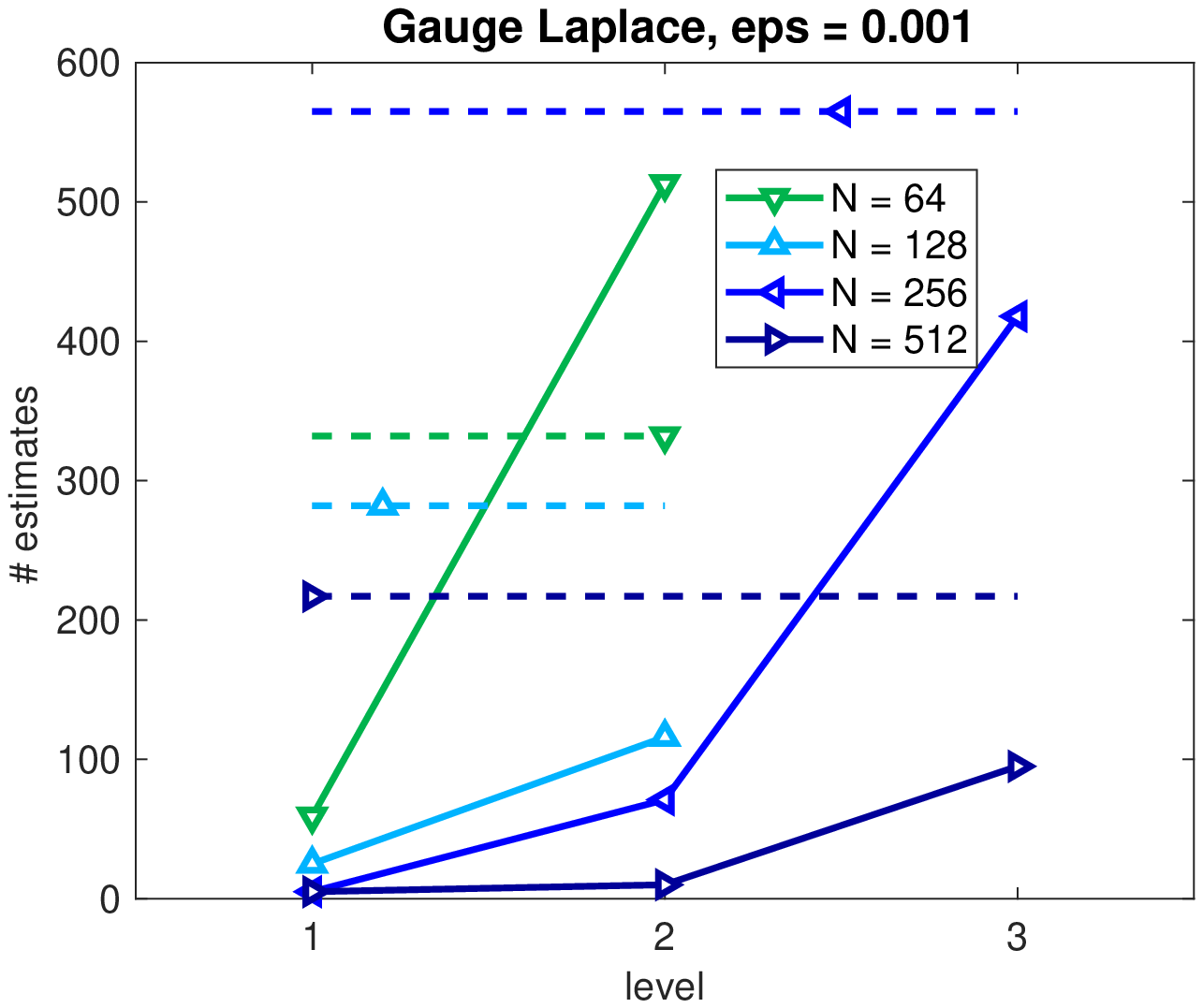} \hfill
\includegraphics[width=.49\textwidth]{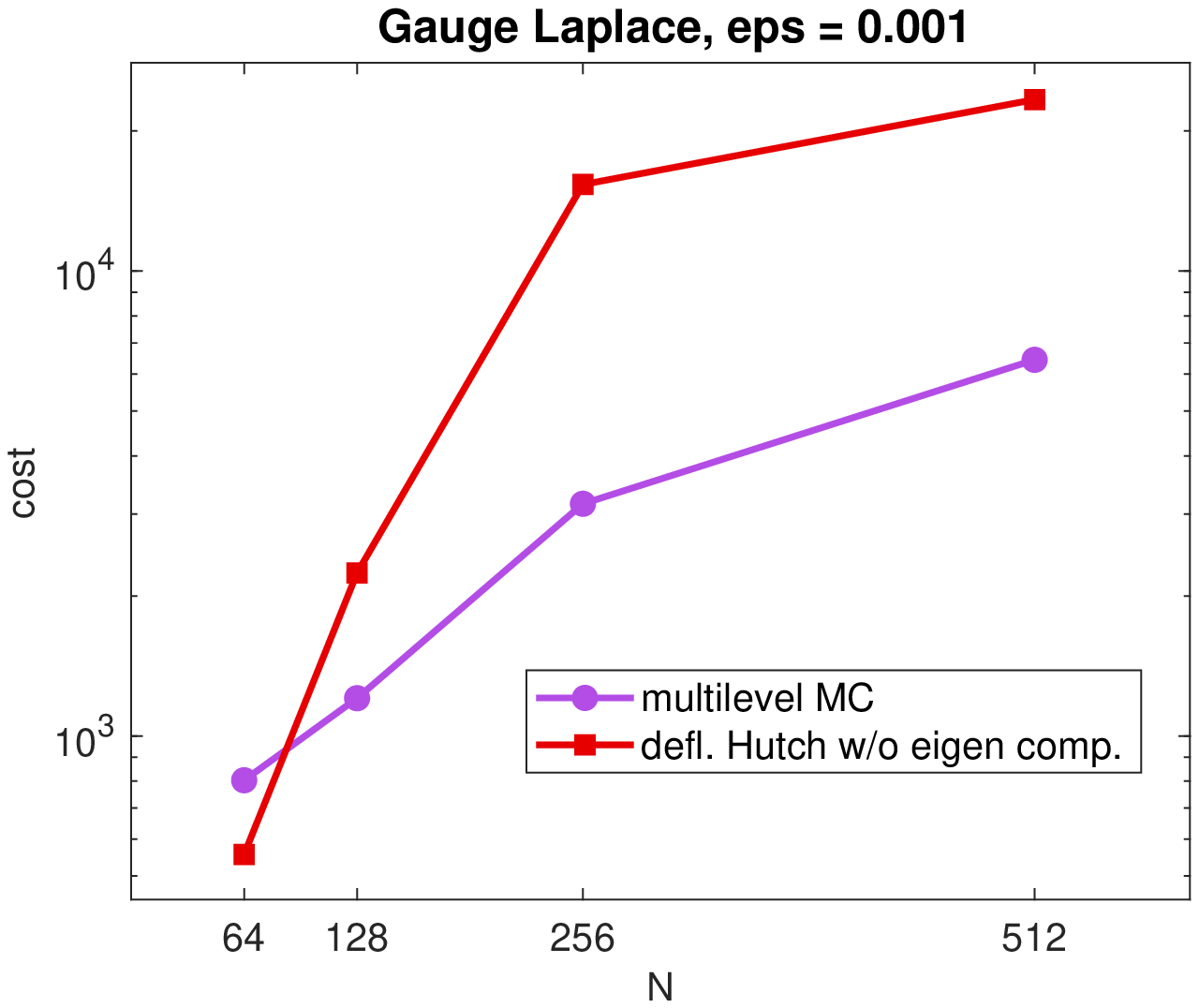}
  \caption{Comparison of multilevel Monte-Carlo and optimally deflated Hutchinson for the gauge Laplace matrices: no of stochastic estimates on each level difference \eqref{eq:trace_diff} and total work for different $N$. \label{fig:gaugeLaplace_estimates_work}}
\end{figure}

In the same manner as Figure~\ref{fig:2dLaplace_estimates_work} did for the 2d Laplacian, Figure~\ref{fig:gaugeLaplace_estimates_work} now reports the number of estimates at the various levels and the arithmetic cost for multilevel Monte-Carlo as well as for optimally deflated Hutchinson. 
 As for the first example the figure illustrates that multilevel Monte-Carlo allows to do few estimates on the fine levels. The gain in work reaches up to a factor of 5 for the larger values of $N$, whereas $N=64$ is too small in this example for multilevel Monte-Carlo to beat deflated Hutchinson. As before, we do not count the arithmetic cost for computing the eigenpairs in deflated Hutchinson in this comparison. 
\end{example}

\begin{example} \label{ex:Schwinger} Our third example is the Schwinger discretization of the 2-di\-men\-si\-onal Dirac operator \cite{Schwinger1962}. This operator describes quantum electrodynamics (QED), the quantum field theory of the electro-magnetic interaction between charged particles via photons. The Schwinger matrix resembles the gauge Laplacian in the sense that there is a periodic nearest neighbor coupling on an equidistant grid in the unit square and that the coupling coefficients are based on complex numbers of modulus 1 with a random phase. The difference is that for Dirac operators there are several (here: two) variables per grid point, representing different spins.  With the Pauli matrices
\[
\sigma_1 = \left[ \begin{array}{cc} 0 & 1 \\ 1 & 0 \end{array} \right], \enspace \sigma_2 = \left[ \begin{array}{cc} 0 & i \\ -i & 0 \end{array} \right],
\]
and the understanding that a grid variable $u_{ij}$ has now two components, i.e.\ $u_{ij} \in \mathC^2$, representing the different spins at grid point $(ih,jh)$,
the periodic couplings is now given as 
\begin{eqnarray} \label{eq:schwinger_coupling}
  (4+m)\cdot u_{ij} &-&  e^{\i\Theta_{ij}}(I-\sigma_1)u_{i+1,j}-e^{\i\Phi_{ij}}(I-\sigma_2)u_{i,j+1} \\
  &-& e^{-\i\Theta_{i-1,j}}(I+\sigma_1)u_{i-1,j}-e^{-\i\Phi_{i,j-1}}(I+\sigma_2)u_{i,j-1}, \\ 
  & &i,j=1,\ldots,N.
\end{eqnarray}

Note that the Pauli matrices cross-couple the spins. Thus, if we order spin components such that the first spin component
at each grid location comes first and all second spin components follow, the Schwinger matrix has the form
\begin{equation} \label{eq:matrix_Schwinger}
S^N = \left[ \begin{array}{cc} G^N & B \\ -B^* & G^N \end{array} \right]
\end{equation}
where the matrices $G^N$ are gauge Laplacians and the matrix $B$ represents the spin cross-coupling. 

We used a Schwinger matrix arising from a thermalized configuration within a Markov process. This guarantees that the random gauge links obey a Boltzmann distribution with a given temperature parameter. The matrix belongs to an $N \times N$ lattice with $N = 128$, and is thus of size $2N^{2}{\times}2N^{2} = 32,768 \times 32,768$.

The multigrid hierarchy for the Schwinger matrix is obtained through an aggregation based approach similar to the one 
typically used for the 4-dimensional Wilson-Dirac operator; see \cite{MGClark2010_1,FroKaKrLeRo13}. Giving all 
details would be beyond the scope of this paper, so here is a rough sketch: At each level, the operator represents a 
periodic nearest neighbor coupling an a 2-dimensional lattice of decreasing size. At each lattice site we have 
several, $d$ say, degrees of freedom (dofs), i.e.\ variables belonging to a lattice site are vectors of length $d$. 
When going from one level to the next, we subdivide the lattice into small sublattices---the aggregates. Each aggregate becomes a single lattice site on the next level. The corresponding restriction operator is obtained by computing (quite inexact) approximations to the $d$ smallest eigenvectors, the components of which are assembled over the aggregates and orthogonalized. This gives 
restriction operators which are orthonormal, and since we take the prolongations to be their adjoints, we are in the simplified situation of Remark~\ref{rem:simplified} for estimating the traces of the differences in multilevel Monte-Carlo.  

The Schwinger matrix is not Hermitian, but its eigenvalues come in complex conjugate pairs. This is due to a non-trivial symmetry induced by the spin structure that can be seen from \eqref{eq:matrix_Schwinger},
\[
 J S^N = \left(S^N\right)^*  J, \enspace \mbox{ where } J = \left[ \begin{array}{cc} I & 0 \\ 0 & -I \end{array} \right],
\]
so that to each right eigenpair $(x,\lambda)$ of $S^N$ corresponds a left eigenpair $((Jx)^*, \bar{\lambda})$. 
This spin symmetry can be preserved on the coarser levels if one doubles the dofs; see \cite{MGClark2010_1,FroKaKrLeRo13}.

We built a multigrid hierarchy with four levels. For the aggregates, at all levels we always aggregated $4 \times 4$ sublattices into one lattice site on the next level, and we started with 2 dofs for the second level and 4 for all remaining levels. Those dofs are then doubled because we implemented the spin structure preserving approach. 
Table~\ref{tab:Schwinger} summarizes the most important information on the multigrid hierarchy. It also shows the five
different (negative) values for the mass $m$ that we used in our experiments. These values are physically meaningful, 
and for all of them the spectrum of $S^N$ is contained in the right half plane. As $m$ becomes smaller, 
$S^N$ 
becomes more ill-conditioned, so the work for each stochastic estimate increases. When solving linear systems at the 
various levels, we used one V-cycle of multigrid with two steps of Gauss-Seidel pre- and post-smoothing as a 
preconditioner for flexible GMRES \cite{Saad2003}. Our implementation was done in Python\footnote{The programs can be found in the GitHub repository https://github.com/Gustavroot/MLMCTraceComputer}. 

\begin{table}
    \centering
    \begin{small}
    \begin{tabular}{|r|l|cccc|c|}
    \hline
    \multicolumn{7}{|c|}{Schwinger model} \\
      \hline
     $N$  &                &$\ell = 1$     &$\ell = 2$    &$\ell = 3$    &$\ell = 4$    & $L$ \\ \hline 
     128   & $n_\ell$      & $2\cdot 128^2$& $4\cdot 32^2$& $8\cdot 8^2$ & $8\cdot2^2$  & 4 \\
          & $\nnz(S^N_\ell)$ & $2.94e5$       & $1.64e5$     & $2.46e4$     & $1024$       &  \\ \hline \hline
    $m$   & \multicolumn{1}{c}{$-0.1320$}      & $-0.1325$     & $-0.1329$    & $-0.1332$    &  $-0.1333$   &   \\
$n_\defl$ & \multicolumn{1}{c}{$384$}  & $384$         & $512$        & $512$        & 512          &      \\ \hline     
    \end{tabular}
    \end{small}
        \caption{Parameters and quantities for Example~\ref{ex:Schwinger}}
    \label{tab:Schwinger}
\end{table}

\begin{figure}[htbp]
\includegraphics[width=.49\textwidth]{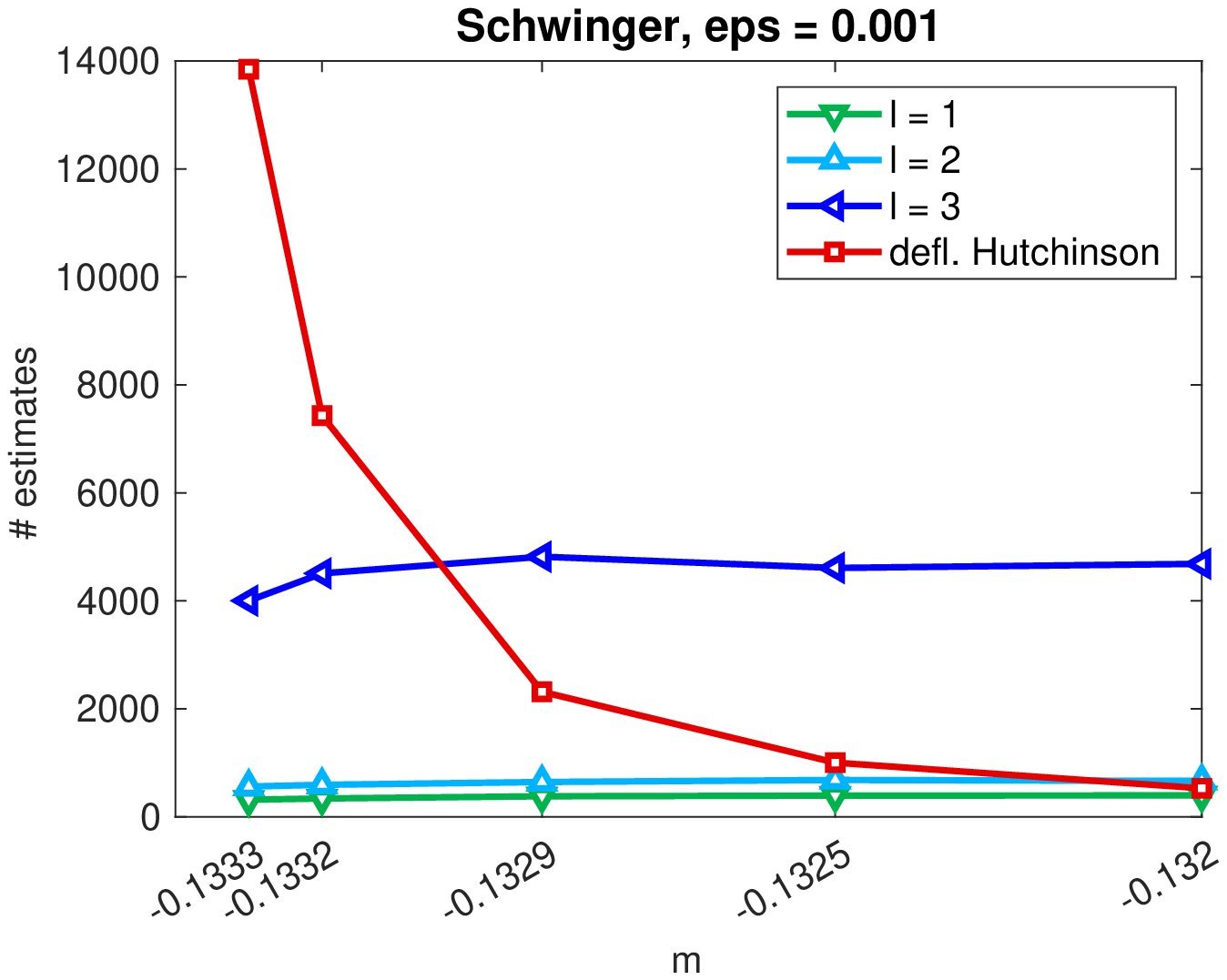} \hfill
\includegraphics[width=.49\textwidth]{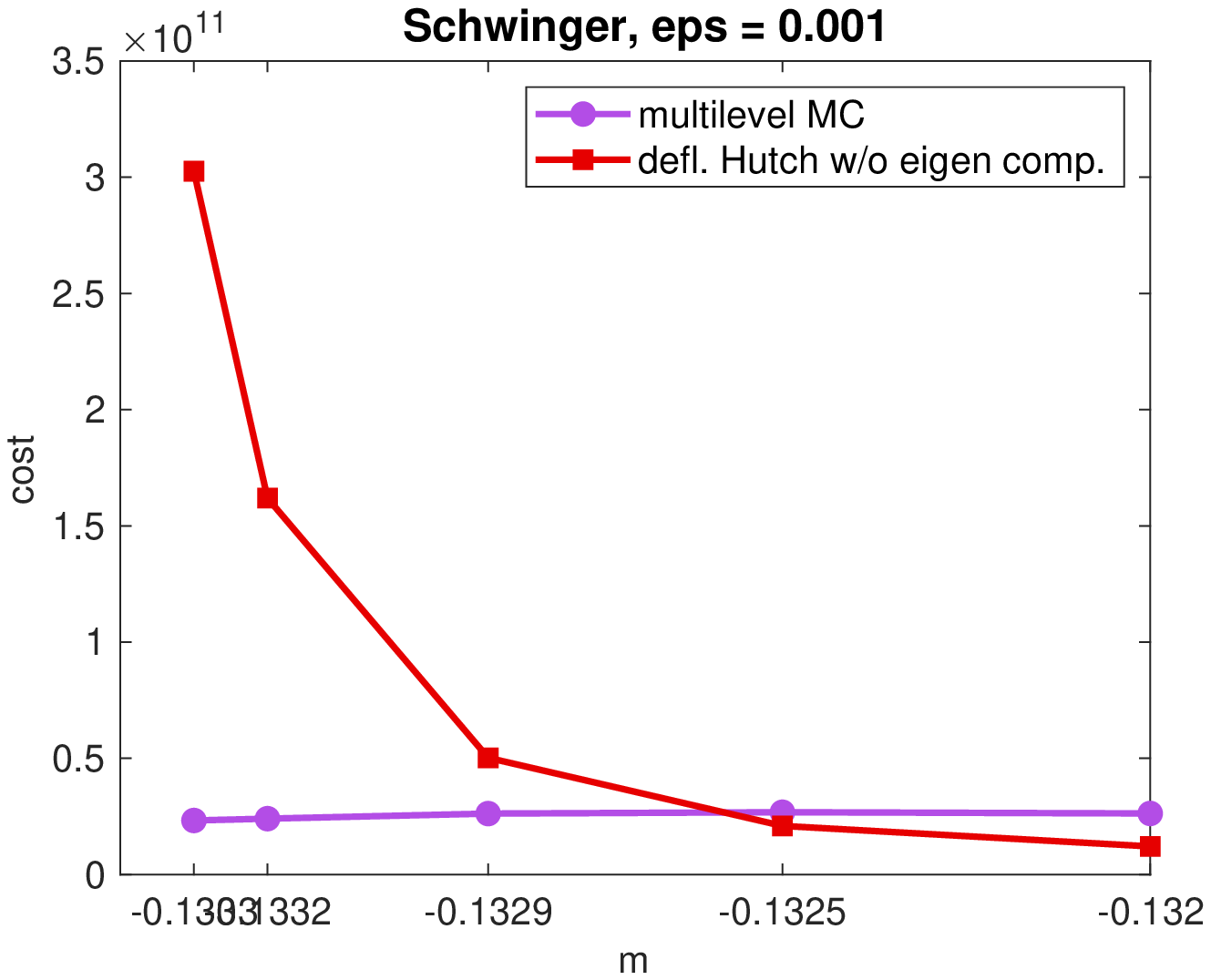}
  \caption{Multilevel Monte-Carlo and deflated Hutchinson for the Schwinger matrix: no of stochastic estimates on each level difference \eqref{eq:trace_diff} and total work for different masses $m$. \label{fig:Schwinger_estimates_work}}
\end{figure}

Figure~\ref{fig:Schwinger_estimates_work} shows our results. We tuned the required root mean square deviation $\rho_\ell$ at each level due to the observation that this time the root mean square deviation is comparably small on the last level difference. The values we chose, independently of the mass parameter $m$, are  $\rho_{1}=\sqrt{0.4}\epsilon\tau$, $\rho_{2}=\sqrt{0.55}\epsilon\tau$ and $\rho_{3} = \sqrt{0.05}\epsilon\tau$ for all masses.

As in the other examples, we compare against deflated Hutchinson with 
a time-optimal number of deflated eigenpairs, and we do not count the work for the eigenpair computation. The figure shows that multilevel Monte-Carlo becomes increasingly more efficient over deflated Hutchinson as the masses become smaller, ending up in a one order of magnitude improvement in work for the smallest. Interestingly, we also see that the number of stochastic estimates to be performed on each level in multilevel Monte-Carlo depends on the masses only very mildly, whereas the number of stochastic estimates increase rapidly in deflated Hutchinson.   
\end{example}

\section*{Conclusion} We presented a novel multilevel Monte-Carlo approach to stochastically estimate the trace of 
a matrix. The method relies on the availability of a multigrid hierarchy and estimates traces of differences of matrices at the various levels. The method is efficient if the variances at the earlier, finer level differences, are small, and this is what we observed in three different examples, two of which used an adaptive algebraic approach to determine the multigrid hierarchy.

\section*{Acknowledgments}
We would like to thank Francesco Knechtli and Tomasz Korzec for stimulating discussion including references on work done in the lattice QCD  community, and Karsten Kahl for providing us with the thermalized Schwinger model matrix.

\bibliographystyle{siamplain}
\bibliography{Paper_mlmc_trace}
\end{document}